\documentclass[12pt]{article}

\usepackage[centertags]{amsmath}
\usepackage{amsthm}
\usepackage{amssymb}
\usepackage{amscd}
\usepackage{eucal}

\newtheoremstyle{my}{1.5em}{0.5em}{\em}{}{\sc}{.}{0.5em}{}
\theoremstyle{my}
\newtheorem{thm}{Theorem}

\newtheorem{remark}[thm]{Remark}
\newtheorem{lemma}[thm]{Lemma}

{

\begin{enumerate} \parsep0em
\itemsep0em \parskip0em}{\parskip-1em \end{enumerate}}

\title{Volume renormalization for conformally compact asymptotically hyperbolic manifolds in dimension four}
\author{Shih-Tsai Feng\\
stfeng@mail.ndhu.edu.tw\\Department of Applied Mathematics\\ National Dong Hwa University, Hualien, Taiwan 97401}

\date{}

\begin{document}

\maketitle
Abstract. In this paper, we derive a new renormalized volume formula for conformally compact asymptotically hyperbolic manifolds in dimesion four. The 
formula generlizes the ones given in Anderson [4], Albin [1], and Chang-Qing-Yang [8] for the case of Poincare-Einstein manifolds. We also derive variational formulas
for the renormalized volume seen as a functional on the manifold.

\section{Introduction}
  Let $M$ be a compact connected oriented manifold with non-empty boundary $\partial M$. Following Penrose, a complete Riemannian metric $g$ on $M$ is 
conformally compact if there is a function $\rho$ on $\overline{M}$ such that $\rho|_{\partial M}=0$, $d\rho \neq 0$ on $\partial M$ and $\rho >0$ on $M$, 
and such that $\overline{g}=\rho^{2}g$ extends at least continuously to a Riemannian metric $\overline{g}$ on $\overline{M}$. Such a function $\rho$ is 
called a boundary defining funciton (bdf) for $\partial M$. We will assume that $\overline{g}$ is at least $C^{3}$ smooth up to $\overline{M}$. Conversely
 if $\overline{g}$ is any smooth Riemannian metric on $\overline{M}$ and if $\rho$ is any $C^{1}$ boundary defining function, then $g=\rho^{-2}\overline{g}$
gives a complete conformally compact metric on $M$.\\

  Note that the defining function is not unique, since any multiple of positive smooth function on $\overline{M}$ gives another defining function. Hence,
both the metric $\overline{g}$ and its induced metric $\gamma$ on $\partial M$ are not uniquely defined by $(M,g)$. However, the conformal class 
$[\gamma]$ of the boundary metric on $\partial M$ is uniquely determined by the complete metric $g$. Because of this, we call $(\partial M,[\gamma])$ the 
conformal infinity of $(M,g)$.\\
 
  For the conformal transformation $\overline{g}=\rho^{2}g$, we have the following formulas for curvatures, see Besse [6].\\
\begin{eqnarray*}
\overline{K}_{ij}&=&\frac{K_{ij}+|\overline{\nabla}\rho|^{2}}{\rho^{2}}-\frac{1}{\rho}(\overline{\nabla}^{2}\rho(\overline{x}_{i},\overline{x}_{i})
+\overline{\nabla}^{2}\rho(\overline{x}_{j},\overline{x}_{j})),\\
\overline{r}&=&r-2\frac{1}{\rho}\overline{\nabla}^{2}\rho-(\frac{\overline{\triangle}\rho}{\rho}-3\frac{|\overline{\nabla}\rho|^{2}}{\rho^{2}})\overline{g},\\
\overline{s}&=&\frac{1}{\rho^{2}}s-6\frac{\overline{\triangle}\rho}{\rho}+12\frac{\rho^{2}}{|\overline{\nabla}\rho|^{2}}.
\end{eqnarray*}

From Riemannian curvature transformation formula, we define the complete metric $g$ to be asymptotically hyperbolic (AH) if the norm of the gradient of 
$\rho$ with respect to the metric $\overline{g}$ is one, i.e. $||\overline{\nabla}\rho||_{\overline{g}}=1$ on $\partial M$. Note that 
$||\overline{\nabla}\rho||_{\overline{g}}$ is an invariant of the conformal infinity $(\partial M,[\gamma])$. We have the following lemma from [14]. See 
lemma 2.1 in [14].

\begin{lemma}[Graham]
If $(M,g)$ is AH, then any metric on $\partial M$ in the conformal class $[\gamma]$ determines a unique boundary defining function $\rho$ in a 
neighborhood of $\partial M$ such that $\overline{g}|_{\partial M}$ is the prescribed boundary metric and such that 
$||\overline{\nabla}\rho||_{\overline{g}}=1$ in the neighborhood.
\end{lemma}

We will call such a boundary defining function, a special bdf. From now on, we will assume that $(M,g)$ is AH, and a special bdf $\rho$ is chosen. By the 
above lemma, in a neighborhood defined by $\rho$, the manifold $M$ is diffeomorphic to $\partial M\times [0,\epsilon)$. If we suppose that the 
compactification is at least $C^{3}$, then in $\partial M\times [0,\epsilon)$, the compactified metric $\overline{g}$ has the form 
\[\overline{g}=d\rho^{2}+g_{\rho}.\]

The metric $g_{\rho}$ can be seen as a family of metrics on $\partial M$ and it has the expansion 
\begin{equation*}
g_{\rho}=g^{(0)}+\rho g^{(1)}+\rho^{2}g^{(2)}+\rho^{3}g^{(3)}+o(\rho^{3})
\end{equation*}

Here $g^{(0)}=\gamma$ is the boundary metric and $g^{(j)}=\frac{1}{j!}\mathcal{L}_{\overline{\nabla}\rho}^{(j)}g_{\rho}|_{\partial M}$ is the j-th Lie derivative. We 
have the following 
\begin{lemma}
The compactification is totally geodesic, namely $(\partial M,\gamma)$ is a totally geodesic submanifold in $(M,\overline{g})$ if and only if $ g^{(1)}=0$.

\end{lemma}

  We will assume from now on that $(M,g)$ is a conformally compact AH manifold of dimension four with a totally geodesic compactification and a special
bdf is chosen. In the case of Poincare-Einstein metric, it is proved in [14], that  $(\partial M,\gamma)$ is totally geodesic. This implies that the 
assumption above is natural.

   In general it is too restrictive to suppose that the metric $\overline{g}$ is $C^{\infty}$ smooth up to the boundary. In the case of 
conformally compact Einstein metrics (Poincare-Einstein metrics), it is well known that there are $\log$ terms in the expansion of $\overline{g}$ near 
$\partial M$, see [14], [1].

  Under the above assumption, the metric has the expansion 

\begin{equation}
g_{\rho}=\gamma+\rho^{2}g^{(2)}+\rho^{3}g^{(3)}+o(\rho^{3})
\end{equation}

in the neighborhood $\partial M\times [0,\epsilon)$. The volume form of the metric $g$ in $\partial M\times [0,\epsilon)$ takes the form 

\begin{equation}
dvol_{g}=(\frac{\det g_{\rho}}{\det \gamma})^{1/2}\frac{d vol_{\gamma}d\rho}{\rho^{4}}.
\end{equation}
 
Here $(\frac{\det g_{\rho}}{\det \gamma})^{1/2}$ has the expansion

\begin{equation}
(\frac{\det g_{\rho}}{\det \gamma})^{1/2}=1+\rho^{2}v^{(2)}+\rho^{3}v^{(3)}+o(\rho^{3}),
\end{equation}

for some functions $v^{(j)}$ on $\partial M$. Note that a simple computation implies that there is no first order term of $\rho$ in the aobve expansion of 
$(\frac{\det g_{\rho}}{\det \gamma})^{1/2}$. Later in section two, we will show that the same is true for any full contraction of the curvature and its 
covariant derivatives.\\

Now consider the asymptotics of the volume $Vol_{g}(\{\rho>\epsilon\})$ as $\epsilon \rightarrow 0$ by integrating $dvol_{g}$ above. We have 

\begin{equation*}
Vol_{g}(\{\rho>\epsilon\})=\int_{\rho>\epsilon}dvol_{g}=C_{0}\epsilon^{-3}+C_{2}\epsilon^{-1}+V+o(1).
\end{equation*}

The coefficients $C_{0}$, $C_{2}$ are integrals over $\partial M$ of local curvature expression of the metric $\gamma$. $V$ is the so-called renormalized
volume, it depends a priori on the choice of the special bdf. \\

In the case of Poincare-Einstein metrics, it is now clear, see Graham [14] Theorem3 .1 or Albin [1] Theorem 2.5, that the renormalized volume is in fact 
independent of the choice of the special bdf. In the present situation, we have

\begin{lemma}
In dimension four, $V$ is independent of the choice of $\gamma$.

\end{lemma}

Renormalized volume and its formula first appeared in the string theory community. They are interested in the conformally compact Einstein metrics because
of the role they play in the so-called AdS/CFT correspoidence. In dimension 4, the renormalized volume has the property that it is a conformal invariant.
The first mathematical proof was given by Graham in [14]. In [4], using the Chern-Gauss-Bonnet theorem with boundary, Anderson first provided a 
differential geometric understanding of the renormalized volume. His result secializes to dimension 4 and in some places the computation still depends 
on the formal
expansion. In Anderson's formula the renormalized volume coupled with the integral of the Weyl tensor which is a conformal invariant. Hence he also 
proved that the renormalized volume is a conformal invariant. Since then, attempts have been made to generalized to higher dimensions and to give a 
rigorous proof. See [1] and [8]. \\
 
In this paper, we show that the reormalized volume formula holds in a more general situation. Namely, we show that in the case $(M,g)$ is 
a conformally compact AH manifold of dimension four with a totally geodesic compactification, there is a well defined renormalized volume. We also compute
the variation of the renormalized vlolume. \\

The content of the paper is the following. In section 2, after recalling some basic facts about double forms, we compute the renormalized volume of 
conformally compact AH manifolds. The proof is in spirit similar to that of Albin's. The differnece is since we do not assume that 
the metric is Einstein, there are odd terms in the expansions of Christoffel symbils, curvatures and their full contractions. This complicates the 
computation. In section 3, we compute directly the first variation of the renormalized volume. \\

Note that in the case of minimal surfaces in three 
dimensional Poincare-Einstein manifolds, Alexikas-Mazzeo [2] considered a situation similar to our case.

\section{Renormalized volume formula}

We begin by recalling some basic facts about double forms, for more details see Labbi [19] and [20], or Gray [17]. Let $(M,g)$ be a Riemannian manifold of 
dimension $n$. Define $\mathcal{D}^{p,q}=\wedge^{p}(M) \otimes \wedge^{q}(M)$ to be the double form of type $(p,q)$. The Kulkarni-Nomizu product is given 
for $a_{1}\otimes a_{2} \in \mathcal{D}^{p,q}$ and $b_{1}\otimes b_{2} \in \mathcal{D}^{r,s}$ by 
\begin{equation*}
(a_{1}\otimes a_{2})\cdot (b_{1}\otimes b_{2})=a_{1}\wedge a_{2}\otimes b_{1}\wedge b_{2}.
\end{equation*}

There is a contraction map $c:\mathcal{D}^{p+1,q+1} \longrightarrow \mathcal{D}^{p,q}$ defined by 
\begin{equation*}
c\omega(x_{1}\wedge\cdots\wedge x_{p},y_{1}\wedge\cdots\wedge y_{q})=\sum_{j=1}^{n}\omega(e_{j}\wedge x_{1}\wedge\cdots\wedge x_{p},e_{j}\wedge y_{1}\wedge\cdots\wedge y_{q})
\end{equation*}
where $\{e_{1},\cdots,e_{n}\}$ is an orthonormal basis. The contraction map has the following property. Let $\omega \in \mathcal{D}^{p,q}$, then 
\begin{equation*}
c(g\cdot \omega)=gc\omega+(n-p-q)\omega.
\end{equation*}

The natural metric $<\cdot,\cdot>$ on $\wedge^{p}(M)$ extends to $\mathcal{D}^{p,q}$ and satisfies 

\begin{equation}
<g\omega_{1},\omega_{2}>=<\omega_{1},c\omega_{2}>.
\end{equation}

If we regard the curvature $R\in \mathcal{D}^{2,2}$ then $cR=r$, the Ricci tensor, and $c^{2}R=cr=s$, the scalar curvature. The Hodge star operator $\ast$ 
extends in a natural way to a linear map $\ast:\mathcal{D}^{p,q}\longrightarrow \mathcal{D}^{n-p,n-q}$, and it satisfies 
\begin{equation}
g\omega=\ast c \ast \omega.
\end{equation}

There are generalized covariant derivatives $D:\mathcal{D}^{p,q}\longrightarrow \mathcal{D}^{p+1,q}$ and 
$\tilde{D}:\mathcal{D}^{p,q}\longrightarrow \mathcal{D}^{p,q+1}$, so that if we denote by $\mathcal{C}^{p}$ the symmetric part of $\mathcal{D}^{p,p}$ 

\begin{equation}
D\tilde{D}+\tilde{D}D:\mathcal{C}^{p}\longrightarrow \mathcal{C}^{p+1}
\end{equation}
is a generalization to $\mathcal{C}^{p}$ of the usual Hessian operator on functions. Define $\delta=c\tilde{D}+\tilde{D}c$ and $\tilde{\delta}=cD+Dc$ then
$\tilde{\delta}\delta+\delta\tilde{\delta}:\mathcal{C}^{p+1}\longrightarrow \mathcal{C}^{p}$ is the formal adjoint of the operator 
$D\tilde{D}+\tilde{D}D$ with respect to the integral saclar product, and we have 
\begin{equation}
\tilde{\delta}\delta+\delta\tilde{\delta}=\ast(D\tilde{D}+\tilde{D}D)\ast.
\end{equation}

Finally for $h \in \mathcal{C}^{1}$, define the self-adjoint operator

\begin{eqnarray}
F_{h}(R)((X,Y)(Z,W))&=&h(R(X,Y)Z,W)-h(R(X,Y)W,Z)+\nonumber \\
                   & &h(R(Z,W)X,Y)-h(R(Z,W)Y,X).
\end{eqnarray}

Then $F_{h}$ satisfies $F_{h}(\omega \theta)=F_{h}(\omega )\theta+\omega F_{h}(\theta)$ and $<F_{h}(\omega ),\theta>=<\omega, F_{h}(\theta)>$.\\

Now we begin our computation of the renormalized volume. At a point $p \in \partial M$, choose coordinate vector fields 
$\{\partial_{x_{i}},\partial_{\rho}\}=\{\overline{X_{s}}\}$ for $\overline{g}$ by exponentiating first on the boundary, then into the manifold. This means
$\{\partial_{x_{i}}\}$ forms a normal coordinate chart for $(\partial M,\gamma)$ cenctered at $p$ and are extended into the interior of $M$ along geodesics
normal to $\partial M$. In this way, with $\rho$ as the fourth coordinate, we have 
\begin{equation}
\overline{g}_{k4}=\delta_{k4}
\end{equation}
$1\leq k \leq 3$. Denote the Christoffel symbols and components of the curvature tensor of $\overline{g}$ in this coordinate chart by 
$\overline{\Gamma}^{k}_{ij}$ and $\overline{R}^{l}_{ijk}$, respectively. We will study the structure of $g$ using the frame $X_{s}:=\rho \overline{X_{s}}$,
$s=1,\cdots,4$. We use the letters $i,j,k$ to denote indices from 1 to 3 and $s,t,u,v$ to denote indices from 1 to 4. Note that 
$[\overline{X_{s}},\overline{X_{s}}]=0$ for all $i,j$, but 
\begin{equation}
[X_{4},X_{i}]=X_{i},   [X_{i},X_{j}]=0
\end{equation}
for $i,j\leq 3$. Also note that $g_{ij}=g(X_{i},X_{j})=\rho^{2}g(\overline{X_{i}},\overline{X_{j}})=\overline{g_{ij}}$. 
The Levi-Civita connection of $g$ is related
to that of $\overline{g}$ by , see Besse[6]

(add equation)

From this we have the following relation between Christoffel symbols 

\begin{lemma}
With $X_{s}=\rho \overline{X_{s}}$, then
\begin{equation}
\Gamma^{u}_{st}=\rho \overline{\Gamma^{u}_{st}}-\delta_{su}\delta_{t4}+\delta_{u4}\overline{g_{st}}.
\end{equation}
From the relation $\overline{g}_{k4}=\delta_{k4}$, we have furthermore 
\begin{equation}
\Gamma^{k}_{i4}=\frac{1}{2}\rho \overline{g}^{kl}\frac{\partial}{\partial \rho}\overline{g}_{li}-\delta_{ik},
\Gamma^{k}_{4i}=\frac{1}{2}\rho \overline{g}^{kl}\frac{\partial}{\partial \rho}\overline{g}_{li},
\Gamma^{4}_{ij}=\overline{g}_{ij}-\frac{1}{2}\rho\frac{\partial}{\partial \rho}\overline{g}_{ij}.
\end{equation}

\end{lemma}

The proof is a simple computation from the formula of Christoffel symbols. See Albin [1].
\begin{remark}
From the lemma, we see immediately that there is no first order term in $\rho$ in the expansion of $\Gamma^{u}_{st}$. Note also that 
  $\Gamma^{u}_{st}=0$ if at least two of $\{s,t,u\}$ are 4.\\
\end{remark}
For curvature, we have the following

\begin{lemma}
Let $(M,g)$ be an AH manifold, and $\rho$ a special bdf. Assume that the metric $\overline{g}$ has an expansion of the form 
\begin{equation}
\overline{g}=d\rho^{2}+\gamma+\rho^{2}g^{(2)}+\rho^{3}g^{(3)}+o(\rho^{3})
\end{equation}
and let $P$ be a full contraction of the curvature and its covariant derivatives.  Then it has a similar expansion

\begin{equation}
P=P^{(0)}+\rho^{2}P^{(2)}+\rho^{3}P^{(3)}+o(\rho^{3}).
\end{equation}

\end{lemma}

\begin{proof}
We only have to consider the curvature tensor $R_{stuv}$, and its first covariant derivatives. The others are similar. Recall the formula
\begin{equation}
R_{stuv}=g_{sw}R^{w}_{tuv}=g_{sw}(\partial_{u}\Gamma^{w}_{tv}-\partial_{v}\Gamma^{w}_{tu}+\Gamma^{n}_{tv}\Gamma^{w}_{nu}-\Gamma^{n}_{tu}\Gamma^{w}_{nv}).
\end{equation}
We divide the consideration into three cases.
\begin{enumerate}
\item If $1\leq s,t,u,v\leq 3$, replace $s,t,u,v$ by $i,j,k,l$, then 
\begin{eqnarray}
R_{ijkl}&=&\sum^{3}_{m=1}g_{im}(\partial_{k}\Gamma^{m}_{jl}-\partial_{l}\Gamma^{m}_{jk}+\Gamma^{n}_{jl}\Gamma^{m}_{nk}-\Gamma^{n}_{jk}\Gamma^{m}_{nl})\nonumber\\
      &=&\sum^{3}_{m=1}\overline{g}_{im}[\rho\overline{\partial}_{k}(\rho\overline{\Gamma}^{m}_{jl})
         -\rho\overline{\partial}_{l}(\rho\overline{\Gamma}^{m}_{jk})+\nonumber\\          
      & &\rho^{2}\sum^{3}_{n=1}(\overline{\Gamma}^{n}_{jl}\overline{\Gamma}^{m}_{nk}-\overline{\Gamma}^{n}_{jk}\overline{\Gamma}^{m}_{nl})+
          \Gamma^{4}_{jl}\Gamma^{m}_{4k}-\Gamma^{4}_{jk}\Gamma^{m}_{4l}]\nonumber
\end{eqnarray}

By the above lemma, there is no first order term in $\rho$ appeared.

\item If one of the $s,t,u,v$ is 4, then we can arrange the indices to reduce to computing
\begin{equation*}
R_{ijk4}=-g_{44}(\partial_{i}\Gamma^{4}_{jk}-\partial_{j}\Gamma^{4}_{ik}+\Gamma^{n}_{jk}\Gamma^{4}_{ni}-\Gamma^{n}_{ik}\Gamma^{4}_{nj}),
\end{equation*}
which has no first order term in $\rho$ by the above lemma about $\Gamma^{4}_{ij}$.

\item If two of $s,t,u,v$ are 4, we compute 
\begin{equation*}
R_{i4k4}=-g_{44}R^{4}_{ik4}=-g_{44}(\partial_{k}\Gamma^{4}_{i4}-\partial_{4}\Gamma^{4}_{ik}+\Gamma^{n}_{i4}\Gamma^{4}_{nk}-\Gamma^{n}_{ik}\Gamma^{4}_{n4}).
\end{equation*}
Now $\Gamma^{4}_{i4}=0$ an $\partial_{4}\Gamma^{4}_{ik}=\rho\partial_{\rho}(\overline{g}_{ik}-1/2\rho\partial_{\rho}\overline{g}_{ik})$, which has no
first order term in $\rho$.

\end{enumerate}
Hence all curvature tensors $R_{stuv}$ has no first order term in $\rho$ in the expansion. This implies all full contractions has no $\rho^{1}$ term. This 
completes the proof.

\end{proof}

For later use, we will need some explicit formulas for the coefficients in the expansion of $\Gamma^{s}_{tu}$ and $R_{stuv}$ in orders of $\rho$. 
It is a simple computation from lemma 4. We list the following that will be used in the proof of the main theorem of this section. We have

\begin{equation}
(\Gamma^{4}_{ij})^{(0)}=\gamma_{ij} 
\end{equation}
\begin{equation}
(\Gamma^{4}_{ij})^{(3)}=-\frac{1}{2}g^{(3)}_{ij} 
\end{equation}
\begin{eqnarray}
R_{ijkl}^{(0)}&=&(\Gamma^{4}_{ik}\Gamma^{4}_{jl}-\Gamma^{4}_{il}\Gamma^{4}_{jk})^{(0)}\nonumber \\
           &=&\gamma_{ik}\gamma_{jl}-\gamma_{il}\gamma_{jk}\\
R_{ijkl}^{(3)}&=&(\Gamma^{4}_{ik}\Gamma^{4}_{jl}-\Gamma^{4}_{il}\Gamma^{4}_{jk})^{(3)}\nonumber \\
           &=&\frac{1}{2}(\gamma_{ik}g^{(3)}_{jl}+\gamma_{jl}g^{(3)}_{ik}-\gamma_{il}g^{(3)}_{jk}-\gamma_{jk}g^{(3)}_{il})
\end{eqnarray}

We also need the relation between $g^{(i)}_{jk}$ and $v^{(i)}$. Recalling that $g^{(j)}=\frac{1}{j!}\mathcal{L}_{\overline{\nabla}\rho}^{(j)}g_{\rho}|_{\partial M}$,
 we have 
\begin{equation}
g^{(3)}_{ij}=-\frac{1}{3}\frac{\partial}{\partial\rho}\overline{R}_{i4j4}|_{\rho=0}
\end{equation}

and
\begin{equation}
v^{(3)}=-\frac{1}{6}\frac{\partial}{\partial\rho}\overline{r}_{4j}|_{\rho=0}
\end{equation}
 Hence in particular,

\begin{equation}
tr_{\gamma}g^{(3)}=2v^{(3)}
\end{equation}

Recall the Chern-Gauss-Bonnet formula from Chern [9]. We have for $\epsilon>0$, $M_{\epsilon}:=\{\rho \geq \epsilon\}$,

\begin{equation}
\int_{M_{\epsilon}}Pff+\int_{\rho=\epsilon}II=\chi(M_{\epsilon})
\end{equation}

where $\chi(M_{\epsilon})$ is the interior Euler characteristic, $Pff$ is the Pfaffian and $II$ is the second fundamental form of $\{\rho=\epsilon\}$,
which is given from Chern [9] in dimension four by 

\begin{equation}
II=\frac{1}{\pi^{2}}(\frac{1}{1\cdot 3 \cdot 2^{2}}\Phi_{0}+\frac{-1}{1\cdot 1 \cdot 2^{3}}\Phi_{1})
\end{equation}

where 

\begin{eqnarray}
\Phi_{0}=\sum_{\sigma \in \mathcal{S}_{3}}\epsilon(\sigma)\omega_{\sigma_{1} 4}\omega_{\sigma_{2} 4}\omega_{\sigma_{3} 4}\\
\Phi_{1}=\sum_{\sigma \in \mathcal{S}_{3}}\epsilon(\sigma)\Omega_{\sigma_{1} \sigma_{2}}\omega_{\sigma_{3}4}
\end{eqnarray}

Here $\omega_{\sigma_{i}4}$ is the connection form and $\Omega_{\sigma_{i}\sigma_{j}}$ is the curvature form.\\

Now for $\epsilon$ sufficiently small, $\chi(M_{\epsilon})=\chi(M)$. The right hand side of (23) is independent of $\epsilon$, neither does the left
hand side, and thus 

\begin{equation}
\sideset{^H}{}\int Pff+FP_{\epsilon=0}\int_{\rho=\epsilon}II=\chi(M),
\end{equation}
where $\sideset{^H}{}\int$ stands for Hadamard regularization.
In terms of the double forms, the Pfaffian of a four dimensional manifold is

\begin{equation}
Pff=\frac{1}{(2\pi)^{2}}\frac{c^{4}R^{2}}{4!2!}dvol.
\end{equation}

A general formula holds for even dimensional manifolds, see Gray [17]. To simplify the expression of $Pff$ in dimension four, we have the orthogonal 
decomposition of curvature tensor, see Besse [6]

\begin{equation*}
R=\frac{s}{24}g\cdot g+\frac{1}{2}z\cdot g +W
\end{equation*}

where $z=r-\frac{s}{4}g$ is the trace free part of the Ricci tensor, and $W$ is the Weyl tensor. So 

\begin{equation}
Pff=\frac{1}{8\pi^{2}}(|W|^{2}-2|r|^{2}+\frac{2}{3}s^{2})dvol_{g}
\end{equation}

Since $|W|^{2}$ is a pointwise conformal invariant, we have in dimension four,

\begin{equation}
\sideset{^H}{}\int Pff=\frac{1}{8\pi^{2}}\int_{M}|W|^{2}dvol_{g}+\frac{1}{12\pi^{2}}\sideset{^H}{}\int_{M}(s^{2}-3|r|^{2})dvol_{g}
\end{equation}

Now we state and prove the main theorem of this section.

\begin{thm}
Let (M,g) be a four dimensional conformall compact AH manifold with special bdf $\rho$ so that in the compactified matric $\overline{g}=\rho^{2} g$,
$\partial M$ is totally geodesic. Then we have
\begin{equation}
\sideset{^H}{}\int Pff=\chi (M),
\end{equation}
or equivalently,
\begin{equation}
\frac{1}{8\pi^{2}}\int_{M}|W|^{2}dvol_{g}+\frac{1}{12\pi^{2}}\sideset{^H}{}\int_{M}(s^{2}-3|r|^{2})dvol_{g}=\chi (M).
\end{equation}

\end{thm}
\begin{proof}
   It suffices to show that the second term in (27) is zero. This is accompllished by a careful computation about the coefficient of the 0th order term 
in the expansion of $\Phi_{0}$ and $\Phi_{1}$.

   For $\Phi_{0}$, we have at $\rho$
\begin{eqnarray*}
\Phi_{0}&=&\sum_{\sigma \in \mathcal{S}_{3}}\epsilon(\sigma)\omega_{\sigma_{1} 4}\omega_{\sigma_{2} 4}\omega_{\sigma_{3} 4}\nonumber\\
       &=&\sum_{\sigma \in \mathcal{S}_{3}}\epsilon(\sigma)\Gamma^{4}_{\sigma_{1}i}\omega^{i}\wedge\Gamma^{4}_{\sigma_{2}j}\omega^{j}
           \wedge\Gamma^{4}_{\sigma_{3}k}\omega^{k}\nonumber\\
       &=&\sum_{\sigma,\eta \in \mathcal{S}_{3}}\epsilon(\sigma)\epsilon(\eta)\Gamma^{4}_{\sigma_{1}\eta_{1}}\Gamma^{4}_{\sigma_{2}\eta_{2}}
           \Gamma^{4}_{\sigma_{3}\eta_{3}}\omega^{1}\wedge\omega^{2}\wedge\omega^{3}\nonumber\\
       &=&\sum_{\sigma,\eta \in \mathcal{S}_{3}}\epsilon(\sigma)\epsilon(\eta)\Gamma^{4}_{\sigma_{1}\eta_{1}}\Gamma^{4}_{\sigma_{2}\eta_{2}}
           \Gamma^{4}_{\sigma_{3}\eta_{3}}dvol_{g_{\rho}}\nonumber\\
       &=&\rho^{-3}\sum_{\sigma,\eta \in \mathcal{S}_{3}}\epsilon(\sigma)\epsilon(\eta)\Gamma^{4}_{\sigma_{1}\eta_{1}}\Gamma^{4}_{\sigma_{2}\eta_{2}}
           \Gamma^{4}_{\sigma_{3}\eta_{3}}(\frac{\det g_{\rho}}{\det \gamma})^{1/2}dvol_{\gamma}\nonumber\\
       &=&\rho^{-3}\sum_{\sigma,\eta \in \mathcal{S}_{3}}\epsilon(\sigma)\epsilon(\eta)\Gamma^{4}_{\sigma_{1}\eta_{1}}\Gamma^{4}_{\sigma_{2}\eta_{2}}
           \Gamma^{4}_{\sigma_{3}\eta_{3}}(1+v^{(2)}\rho^{2}+v^{(3)}\rho^{3}+\cdots)dvol_{\gamma}\nonumber\\
\end{eqnarray*}

The coefficient in the constant term of the expansion of $\Phi_{0}$ is
\begin{eqnarray*}
& &\sum_{\sigma,\eta \in \mathcal{S}_{3}}\epsilon(\sigma)\epsilon(\eta)(\Gamma^{4}_{\sigma_{1}\eta_{1}}\Gamma^{4}_{\sigma_{2}\eta_{2}}
           \Gamma^{4}_{\sigma_{3}\eta_{3}})^{(3)}+v^{(3)}\sum_{\sigma,\eta \in \mathcal{S}_{3}}\epsilon(\sigma)\epsilon(\eta)(\Gamma^{4}_{\sigma_{1}\eta_{1}}
           \Gamma^{4}_{\sigma_{2}\eta_{2}}\Gamma^{4}_{\sigma_{3}\eta_{3}})^{(0)}\\
&=&\sum_{\sigma,\eta \in \mathcal{S}_{3}}\epsilon(\sigma)\epsilon(\eta)(\Gamma^{4(3)}_{\sigma_{1}\eta_{1}}\Gamma^{4(0)}_{\sigma_{2}\eta_{2}}
        \Gamma^{4(0)}_{\sigma_{3}\eta_{3}}+\Gamma^{4(0)}_{\sigma_{1}\eta_{1}}\Gamma^{4(3)}_{\sigma_{2}\eta_{2}}
        \Gamma^{4(0)}_{\sigma_{3}\eta_{3}}+\Gamma^{4(0)}_{\sigma_{1}\eta_{1}}\Gamma^{4(0)}_{\sigma_{2}\eta_{2}}
        \Gamma^{4(3)}_{\sigma_{3}\eta_{3}})\\
& &+v^{(3)}\sum_{\sigma,\eta \in \mathcal{S}_{3}}\epsilon(\sigma)\epsilon(\eta)\Gamma^{4(0)}_{\sigma_{1}\eta_{1}}\Gamma^{4(0)}_{\sigma_{2}\eta_{2}}
           \Gamma^{4(0)}_{\sigma_{3}\eta_{3}}\\
&=&-\frac{3}{2}\sum_{\sigma,\eta \in \mathcal{S}_{3}}\epsilon(\sigma)\epsilon(\eta)\gamma_{\sigma_{1}\eta_{1}}\gamma_{\sigma_{2}\eta_{2}}g^{(3)}_{\sigma_{3}\eta_{3}}
          +\frac{1}{6}v^{(3)}c^{3}(\gamma^{3})\\
&=&-\frac{3}{2}(\frac{1}{6}c^{3}(\gamma^{2}g^{(3)}))+\frac{1}{6}v^{(3)}c^{3}(\gamma^{3})\\
&=&-3tr_{\gamma}g^{(3)}+6v^{(3)}\\
&=&0
\end{eqnarray*}

where in the last three equalities, we have used the expansion of $\Gamma^{4}_{ij}$ and a formula about consecutive application of the contraction map
$c$ on double forms. See Labbi [19].\\
For $\Phi_{1}$, we have at $\rho$
\begin{eqnarray*}
\Phi_{1}&=&\sum_{\sigma \in \mathcal{S}_{3}}\epsilon(\sigma)\Omega_{\sigma_{1}\sigma_{2}}\omega_{\sigma_{3}4}\\
       &=&-\frac{1}{2}\sum_{\sigma \in \mathcal{S}_{3}}\epsilon(\sigma)R_{\sigma_{1}\sigma_{2}ij}\omega^{i}\wedge\omega^{j}\wedge\Gamma^{4}_{\sigma_{3}k}\omega^{k}\\
       &=&-\frac{1}{2}\sum\epsilon(\sigma)\epsilon(\eta)R_{\sigma_{1}\sigma_{2}\eta_{1}\eta_{2}}\Gamma^{4}_{\sigma_{3}\eta_{3}}
          \omega^{\eta_{1}}\wedge\omega^{\eta_{2}}\wedge\omega^{\eta_{3}}\\
       &=&-\frac{1}{2}\rho^{-3}\sum\epsilon(\sigma)\epsilon(\eta)R_{\sigma_{1}\sigma_{2}\eta_{1}\eta_{2}}\Gamma^{4}_{\sigma_{3}\eta_{3}}
          (1+v^{(2)}\rho^{2}+v^{(3)}\rho^{3}+\cdots)dvol_{\gamma}\\
\end{eqnarray*}

The coefficient in the constant term of the expansion of $\Phi_{1}$ is
\begin{eqnarray*}
& &-\frac{1}{2}[\sum\epsilon(\sigma)\epsilon(\eta)(R_{\sigma_{1}\sigma_{2}\eta_{1}\eta_{2}}\Gamma^{4}_{\sigma_{3}\eta_{3}})^{(3)}
                +v^{(3)}\sum\epsilon(\sigma)\epsilon(\eta)(R_{\sigma_{1}\sigma_{2}\eta_{1}\eta_{2}}\Gamma^{4}_{\sigma_{3}\eta_{3}})^{(0)}]\\
&=&-\frac{1}{2}[\sum\epsilon(\sigma)\epsilon(\eta)(R_{\sigma_{1}\sigma_{2}\eta_{1}\eta_{2}}^{(0)}\Gamma^{4(3)}_{\sigma_{3}\eta_{3}}
                 +R_{\sigma_{1}\sigma_{2}\eta_{1}\eta_{2}}^{(3)}\Gamma^{4(0)}_{\sigma_{3}\eta_{3}})
                +v^{(3)}\sum\epsilon(\sigma)\epsilon(\eta)R_{\sigma_{1}\sigma_{2}\eta_{1}\eta_{2}}^{(0)}\Gamma^{4(0)}_{\sigma_{3}\eta_{3}}]\\
\end{eqnarray*}

Utilizing (16)-(19), the above equals to 

\begin{eqnarray*}
&=&-\frac{1}{4}\sum\epsilon(\sigma)\epsilon(\eta)[(\gamma_{\sigma_{1}\eta_{1}}\gamma_{\sigma_{2}\eta_{2}}-\gamma_{\sigma_{1}\eta_{2}}\gamma_{\sigma_{2}\eta_{1}})
                                             g^{(3)}_{\sigma_{3}\eta_{3}}\\
& &+(\gamma_{\sigma_{1}\eta_{1}}g^{(3)}_{\sigma_{2}\eta_{2}}+\gamma_{\sigma_{2}\eta_{2}}g^{(3)}_{\sigma_{1}\eta_{1}}-
                  \gamma_{\sigma_{1}\eta_{2}}g^{(3)}_{\sigma_{2}\eta_{1}}-\gamma_{\sigma_{2}\eta_{1}}g^{(3)}_{\sigma_{1}\eta_{2}})\gamma_{\sigma_{3}\eta_{3}}]\\
& &+\frac{1}{2}v^{(3)}\sum\epsilon(\sigma)\epsilon(\eta)(\gamma_{\sigma_{1}\eta_{1}}\gamma_{\sigma_{2}\eta_{2}}-\gamma_{\sigma_{1}\eta_{2}}\gamma_{\sigma_{2}\eta_{1}})
                   \gamma_{\sigma_{3}\eta_{3}}\\
&=&-\frac{3}{2}\sum\epsilon(\sigma)\epsilon(\eta)\gamma_{\sigma_{1}\eta_{1}}\gamma_{\sigma_{2}\eta_{2}}g^{(3)}_{\sigma_{3}\eta_{3}}\\
& &+v^{(3)}\sum\epsilon(\sigma)\epsilon(\eta)\gamma_{\sigma_{1}\eta_{1}}\gamma_{\sigma_{2}\eta_{2}}\gamma_{\sigma_{3}\eta_{3}}\\
&=&-\frac{3}{2}(\frac{1}{6}c^{3}(\gamma^{2}g^{(3)}))+\frac{1}{6}v^{(3)}c^{3}(\gamma^{3})\\
&=&-3tr_{\gamma}g^{(3)}+6v^{(3)}\\
&=&0
\end{eqnarray*}

So the coefficient in the constant term of the expansion of $II$ is zero. Hence $FP_{\epsilon=0}\int_{\rho=\epsilon}II=0$. This completes the proof.
\end{proof}

\begin{remark}
Expanding the integrand $(s^{2}-3|r|^{2})dvol_{g}$ of the renormalized integral, since $g$ is AH, the zeroth order term is a constant. This shows that the 
above theorem gives a formula for the renormalized volume for the AH metric $g$. Furthermore, if (M,g) is assumed to be Poincare-Einstein with 
normalized Ricci tensor $r=-3g$, then $s^{2}-3|r|^{2}=36$, the above formula coincides with that of Albin's 
and differs by a constant to that of Anderson's.
\end{remark}

\section{Varying the asymptotically hyperbolic metric}

In this section we consider the first variation of the renormalized integral appeared in Theorem 7 under the assumption that 
the boundary metric $\gamma$ is fixed. Instead of doing the computation directly, we will first make some modification. Recalling the trace free part of
the Ricci tensor $z=r-\frac{s}{4}g$, it has norm $|z|^{2}=|r|^{2}-\frac{1}{4}s^{2}$. Hence $s^{2}-3|r|^{2}=\frac{1}{4}s^{2}-3|z|^{2}$. The formula in Theorem 7 
is then rewritten as 
\begin{equation}
\frac{1}{8\pi^{2}}\int_{M}|W|^{2}dvol_{g}+\frac{1}{48\pi^{2}}\sideset{^H}{}\int_{M}s^{2}dvol_{g}-\frac{1}{4\pi^{2}}\sideset{^H}{}\int_{M}|z|^{2}dvol_{g}=\chi (M).
\end{equation}

We will consider the variation of the intrgral $\mathcal{Z}=\sideset{^H}{}\int_{M}|z|^{2}dvol_{g}$, so that the critical points minimize 
the functional 
\begin{equation}
\frac{1}{8\pi^{2}}\int_{M}|W|^{2}dvol_{g}+\frac{1}{48\pi^{2}}\sideset{^H}{}\int_{M}s^{2}dvol_{g}.
\end{equation}

This means that the above funcitonal has a lower bound $\chi (M)$, a topological invariant, at the critical points.

  We need the following formulas for derivatives in the direction $h$ of various 
curvatures expressed in double form.

\begin{lemma} 

\begin{eqnarray}
R'h&=&-\frac{1}{4}(D\tilde{D}+\tilde{D}D)h+\frac{1}{4}F_{h}R\\
r'h&=&-\frac{1}{4}c(D\tilde{D}+\tilde{D}D)h-\frac{1}{4}cF_{h}R+\frac{1}{2}r\circ h+\frac{1}{2}h\circ r\\
s'h&=&-\frac{1}{4}c^{2}(D\tilde{D}+\tilde{D}D)h-<r,h>
\end{eqnarray}
\end{lemma}

\begin{proof}
The first equation is lemma 4.1 in [20]. It can also be computed directly from the formula in Besse [6]. The third equation is obtained by taking 
contraction twice of $R'h$ and by taking into account of the derivative of $g$. Note that $s'h$ also appeared in [20], as $h'_{2}h$ in the main theorem.
Now we prove the second equation.\\
Recall that, the derivatives $\nabla'h$, $R'h$ at $g$ of the Levi-Civita connection and the $(3,1)$-Riemannian curvature tensor are respectively given
by [6]
\begin{eqnarray*}
g(\nabla'h(x,y),z)&=&\frac{1}{2}\{\nabla_{x}h(y,z)+\nabla_{y}h(x,z)-\nabla_{z}h(x,y)\},\\
R'h(x,y)z&=&(\nabla_{y}\nabla'h)(x,z)-(\nabla_{x}\nabla'h)(y,z).
\end{eqnarray*}
By definition, 
\begin{equation*}
r'h(x,y)=tr(z\longmapsto R'h(x,z)y).
\end{equation*} 
Taking an orthonormal frame $\{x_{i}\}$ and noting that differentiation commutes with taking trace, we have 
\begin{eqnarray*}
r'h(x,y)&=&g(R'h(x,x_{i})y,x_{i})\\
        &=&g(\nabla_{x_{i}}\nabla'h(x,y),x_{i})-g(\nabla_{x}\nabla'h(x_{i},y),x_{i})\\
        &=&\frac{1}{2}[\nabla^{2}_{x_{i}x}h(y,x_{i})+\nabla^{2}_{x_{i}y}h(x,x_{i})-\nabla^{2}_{x_{i}x_{i}}h(x,y)\\
        & &-\nabla^{2}_{xx_{i}}h(y,x_{i})-\nabla^{2}_{xy}h(x_{i},x_{i})+\nabla^{2}_{xx_{i}}h(x_{i},y)]\\
        &=&\frac{1}{2}[\nabla^{2}_{x_{i}y}h(x,x_{i})+\nabla^{2}_{xx_{i}}h(x_{i},y)-\nabla^{2}_{x_{i}x_{i}}h(x,y)-\nabla^{2}_{xy}h(x_{i},x_{i})\\
        & &-h(R(x,x_{i})x_{i},y)-h(R(x,x_{i})y,x_{i})]\\
        &=&-\frac{1}{2}D\tilde{D}h(x\wedge x_{i},y\wedge x_{i})-\frac{1}{2}[h(R(x,x_{i})x_{i},y)+h(R(x,x_{i})y,x_{i})]\\
\end{eqnarray*}
Here we have used the identity
\begin{equation*}
\nabla^{2}_{xy}h(z,u)-\nabla^{2}_{yx}h(z,u)=h(R(x,y)z,u)+h(R(x,y)u,z).
\end{equation*}
Now symmetrize the above expression and recall the definition of $F_{h}R$ to get
\begin{eqnarray*}
r'h(x,y)&=&-\frac{1}{4}(D\tilde{D}+\tilde{D}D)h(x\wedge x_{i},y\wedge x_{i})\\
        & &-\frac{1}{4}[h(R(x,x_{i})y,x_{i})+h(R(x,x_{i})x_{i},y)+h(R(y,x_{i})x,x_{i})+h(R(y,x_{i})x_{i},x)]\\
        &=&-\frac{1}{4}(D\tilde{D}+\tilde{D}D)h(x\wedge x_{i},y\wedge x_{i})-\frac{1}{4}F_{h}R(x\wedge x_{i},y\wedge x_{i})\\
        & &+\frac{1}{2}h(R(x_{i},x)x_{i},y)+\frac{1}{2}h(R(x_{i},y)x_{i},x)
\end{eqnarray*}
This completes the proof.
\end{proof}

We also need the following 
\begin{lemma}For any $z\in \mathcal{C}^{1}$
\begin{eqnarray}
<z,cF_{h}R>&=&<g \cdot z,F_{h}R>=<F_{h}(g\cdot z),R>\nonumber\\
          &=&2<h \cdot z,R>+<g\cdot F_{h}z,R>\nonumber\\
          &=&8<\overset{\circ}{R}(z),h>+2<r\circ z,h>\\
<z,\overset{\circ}{R}(h)>&=&<\overset{\circ}{R}(z),h>
\end{eqnarray}
where $\overset{\circ}{R}(z)(x,y)=\sum_{i=1}^{4}z(R(x,x_{i})y,x_{i})$.
\end{lemma}

Next we define for $\omega \in \mathcal{C}^{2}$, the generalized Einstein tensor $T_{2}(\omega)$ by 
\begin{equation}
T_{2}(\omega)=\frac{1}{2}c^{2}\omega g-c\omega.
\end{equation}

In particular, if $\omega=R$, $T_{2}(R)$ is the usual Einstein tensor. We also remark that there is a general definition for $T_{2q}$ for $(1,q)$-curvature
tensor, see Labbi[19]. Now we prove the main theorem of this section.

\begin{thm}
The Euler-Lagrange equation for the functional $\mathcal{Z}=\sideset{^H}{}\int_{M}|z|^{2}dvol_{g}$ on the space of conformally compact AH metrics on $M$ 
is 

\begin{equation}
T_{2}((D\tilde{D}+\tilde{D}D)z)=2f
\end{equation}
where $f=\frac{1}{2}|z|^{2}g-4\overset{\circ}{R}(z)-r\circ z$.
\end{thm}

\begin{proof}
 Let $g(s)$ be a family of conformally compact AH metrics on $M$ with fixed boundary, and $\rho$ be a special bdf for $g=g(0)$. The metric 
$g(s)^{(0)}=\rho^{2}g(s)|_{\partial M}$
then defines unique special bdfs $\rho(s)=\rho e^{w(s,\rho)}$ with respect to $g(s)$. (lemma 2.1 of [14]) We suppose that the metrics 
$\overline{g}(s)=\rho^{2}g(s)$ all have totally geodesic boundaries. Let $h=\frac{d}{ds}|_{s=0}\overline{g}(s)$, then $h|_{\partial M}=0$ and 
$\frac{\partial h}{\partial \rho}|_{\partial M}=0$. Let $Z=|z|^{2}dvol_{g}$, $S=s^{2}dvol_{g}$ and 
$N=|r|^{2}dvol_{g}=<r,r>dvol_{g}$, then $Z=N-\frac{1}{4}S$. 

We begin with the computation of $Z'h$. For $S'h$, we have 
\begin{equation}
S'h=(2ss'h+\frac{1}{2}s^{2}<g,h>)dvol_{g},
\end{equation}

where the second term on the right hand side comes from the derivative of the volume form $dvol_{g}$. For $N'h$, we have
\begin{equation}
N'h=(2<r,r'h>-2<r\circ r,h>+\frac{1}{2}|r|^{2}<g,h>)dvol_{g},
\end{equation}

where the second term on the right hand side is the derivative of the metric and the third term is the derivative of the volume form. So

\begin{eqnarray}
Z'h&=&N'h-\frac{1}{4}S'h\nonumber\\
   &=&(\frac{1}{2}(|r|^{2}-\frac{1}{4}s^{2})<g,h>-2<r\circ r,h>+2<r,r'h>-\frac{1}{2}ss'h)dvol_{g}\nonumber\\
   &=&(\frac{1}{2}|z|^{2}<g,h>-\frac{1}{2}<r,cF_{h}R>+\frac{1}{2}<sr,h>\nonumber\\
   & &-\frac{1}{2}<z\cdot g,(D\tilde{D}+\tilde{D}D)h>)dvol_{g}\nonumber\\
   &=&(\frac{1}{2}|z|^{2}<g,h>-4<\overset{\circ}{R}(z),h>-<r\circ z,h>\nonumber\\
   & &-\frac{1}{2}<z\cdot g,(D\tilde{D}+\tilde{D}D)h>)dvol_{g}
\end{eqnarray}

Let 
\begin{equation}
f=\frac{1}{2}|z|^{2}g-4\overset{\circ}{R}(z)-r\circ z,
\end{equation}

then

\begin{eqnarray}
\mathcal{Z}'h&=&\frac{d}{ds}|_{s=0}\sideset{^R}{}\int_{M}Z=\frac{d}{ds}|_{s=0}\sideset{^R}{}\int_{M}|z|^{2}dvol_{g}\nonumber\\
         &=&\sideset{^R}{}\int_{M}<f,h>dvol_{g}-\frac{1}{2}\sideset{^R}{}\int_{M}<z\cdot g,(D\tilde{D}+\tilde{D}D)h>dvol_{g}
\end{eqnarray}

Since $h$ satisfies $h|_{\partial M}=0$ and $\frac{\partial h}{\partial \rho}|_{\partial M}=0$, an easy argument involving cut-off functions shows that 
\begin{equation*}
<z\cdot g,(D\tilde{D}+\tilde{D}D)h>=<(\delta\tilde{\delta}+\tilde{\delta}\delta)(z\cdot g),h>.
\end{equation*}

Now using (7) and a general formula of Labbi about Hodge star operator, we have $(\delta\tilde{\delta}+\tilde{\delta}\delta)(z\cdot g)=-T_{2}(\omega)$, 
where $\omega=(D\tilde{D}+\tilde{D}D)z$. To find the Euler-Lagrange equation of the functional $\mathcal{Z}$, we have to solve 
\begin{equation}
\sideset{^H}{}\int_{M}<f-\frac{1}{2}T_{2}(\omega),h>dvol_{g}=0
\end{equation}

Since $h$ is a variational field, we may divide the argument into two parts. If $h$ is supported inside $M_{\epsilon}$ for some $\epsilon >0$, the above 
integral has no sigularity and the Euler-Lagrange equation for the above functional is $T_{2}(\omega)=2f$. 
Now let $h$ be supported in $M\setminus M_{\epsilon}$. We may suppose that $\epsilon$ is small enough so that $M\setminus M_{\epsilon}$ has the product 
structure $[0,\epsilon)\times \partial M$. In $[0,\epsilon)\times \partial M$, $\overline{g}(s)$ has the expansion

\begin{equation*}
\overline{g}(s)=d\rho^{2}(s)+g^{(0)}(s)+g^{(2)}(s)\rho^{2}(s)+g^{(3)}(s)\rho^{3}(s)+o(\rho^{3}(s))
\end{equation*}
and $h$ has the expansion

\begin{equation*}
h=d\rho^{2}+h^{(2)}\rho^{2}+h^{(3)}\rho^{3}+o(\rho^{3})
\end{equation*}
Note that $h$ has no constant term because the family $g(s)$ has fixed boundary and $h$ has no first order term because of the form of 
$\overline{g}(s)$. We will solve, in terms of the $\rho$-independent frame $X_{u}=\rho \overline{X}_{u}$, the equation (47). First we write it as

\begin{eqnarray*}
0&=&FP\int_{0}^{\epsilon}\int_{\partial M}<f-\frac{1}{2}T_{2}(\omega),h>dvol_{g}\\
 &=&FP\int_{0}^{\epsilon}\rho^{-4}\int_{\partial M}<f-\frac{1}{2}T_{2}(\omega),h>(\frac{\det g_{\rho}}{\det \gamma})^{1/2}dvol_{\gamma}d\rho
\end{eqnarray*}

Let $\phi(\rho)=\int_{\partial M}<f-\frac{1}{2}T_{2}(\omega),h>(\frac{\det g_{\rho}}{\det \gamma})^{1/2}dvol_{\gamma}$, then $\phi(\rho)$ has the expansion
\begin{equation*}
\phi(\rho)=\sum_{k=0}^{\infty}\frac{\phi^{(k)}(0)}{k!}\rho^{k}.
\end{equation*}
In this way, the equation is reduced to 

\begin{equation*}
0=FP\int_{0}^{\epsilon}\rho^{-4}\phi(\rho)d\rho.
\end{equation*}

Now a computation of Paycha [23] then shows that the finite part on the right hand side is 

\begin{eqnarray*}
& &FP\int_{0}^{\epsilon}\rho^{-4}\phi(\rho)d\rho\\
&=&\frac{\phi^{(3)}(0)}{3!}\sum^{3}_{j=1}\frac{1}{j}-\frac{1}{3!}\int^{\epsilon}_{0}\log \rho\phi^{(4)}(\rho)d\rho\\
&=&\frac{\phi^{(3)}(0)}{3!}\sum^{3}_{j=1}\frac{1}{j}-\sum_{r=0}^{\infty}\phi^{(4+r)}(0)f_{r}(\epsilon),
\end{eqnarray*}
where $f_{r}(\epsilon)$ is a function in $\epsilon$ with leading term $\epsilon^{r+1}\log \rho$. Since $\epsilon$ is an arbitrary parameter, we must have
\begin{equation*}
\phi^{(k)}(0)=0
\end{equation*}
for $k\geq 3$.\\

To procede further, set $E=f-\frac{1}{2}T_{2}(\omega)$, the integrand of $\phi(\rho)$
has the following expression

\begin{equation*}
\phi^{(k)}(0)=\sum^{k}_{i=0}<E,h>^{(i)}v^{(k-i)}
\end{equation*}

where $(\frac{\det g_{\rho}}{\det \gamma})^{1/2}$ is expanded as 
\begin{equation*}
(\frac{\det g_{\rho}}{\det \gamma})^{1/2}=1+\rho^{2}v^{(2)}+\rho^{3}v^{(3)}+o(\rho^{3}),
\end{equation*}

and $<E,h>$ is expanded as

\begin{equation*}
<E,h>=\sum^{\infty}_{i=0}<E,h>^{(i)}\rho^{i}.
\end{equation*}

Now the integrand of $\phi^{(0)}(0)$ is

\begin{equation*}
<E,h>^{(0)}=<E^{(0)},h^{(0)}>_{\gamma}=0
\end{equation*}

by the assumption about $h$. Hence $\phi^{(0)}(0)=0$. The integrand of $\phi^{(1)}(0)$ is 
\begin{equation*}
<E,h>^{(1)}=<E^{(0)},h^{(1)}>_{\gamma}+<E^{(1)},h^{(0)}>_{\gamma}=0
\end{equation*}

again by the assumption about $h$. Hence $\phi^{(1)}(0)=0$. The integrand of $\phi^{(2)}(0)$ is 
\begin{equation*}
<E,h>^{(2)}=<E^{(0)},h^{(2)}>_{\gamma},
\end{equation*}

and the integrand of $\phi^{(3)}(0)$ is 
\begin{equation*}
<E,h>^{(3)}=<E^{(0)},h^{(3)}>_{\gamma}+<E^{(1)},h^{(2)}>_{\gamma}.
\end{equation*}
So $\phi^{(3)}(0)=0$ implies that $E^{(0)}=0$ and $E^{(1)}=0$. And hence $\phi^{(2)}(0)=0$.

The integrand of $\phi^{(4)}(0)$ is 

\begin{equation*}
<E^{(2)},h^{(2)}>_{\gamma}.
\end{equation*}

So $\phi^{(4)}(0)=0$ implies that $E^{(2)}=0$. Solving the equation successively, we have $E^{(k)}=0$ for $k\geq 0$. Namely, we have 
\begin{equation*}
T_{2}(\omega)=2f.
\end{equation*}

This completes the proof.

\end{proof}

\bibliographystyle{amsplain}

\begin{thebibliography}{1}



\bibitem{}
 P. Albin, \emph{Renormalizing curvature integrals on Poincare-Einstein manifolds}, Adv. Math. \textbf{221}(2009), 140-169

\bibitem{}
S. Alexakis, R. Mazzeo, \emph{Renormalized area and properly embedded minimal surfaces in hyperbolic 3-manifolds}, preprint, 2008,
                        arXiv:0802.2250
\bibitem{}
C.B. Allendoerfer, A. Weil, \emph{The Gauss-Bonnet theorem for Riemannian polyhedra}, Trans. Amer. Math. Soc. \textbf{53} (1943) 101-129


\bibitem{}
M.T. Anderson, \emph{$L^{2}$ curvature and volume renormalization of AHE metrics on 4-manifolds}, Math. Res. Lett. \textbf{8}(2001), 171-188

\bibitem{}
M.T. Anderson, \emph{Some results on the structure of conformally compact Einstein metrics}, preprint, 2005, arXiv: math.DG/0402198

\bibitem{}
A.L. Besse, \emph{Einstein manifolds}, Springer-Verlag (1987)

\bibitem{}
O. Biquard, \emph{Metriques d'Einstein asymptotiquement symetriques}, Asterique \textbf{265}(2000) 

\bibitem{}
S.Y.A. Chang, J. Qing, P. Yang, \emph{Renormalized volumes for conformally compact Einstein manifolds}, Sovrem. Mat. Fundam. Napravl. \textbf{17}
                               (2006), 129-142

\bibitem{}
S.S. Chern, \emph{On the curvature integra in a Riemannian manifold}, Ann. of Math. \textbf{46}(1945) 674-684.

\bibitem{}
P.T. Chrusciel, E. Delay, J.M. Lee, D.N. Skinner, \emph{Boundary regularity of conformally compact Einstein metrics}, JDG \textbf{69}(2005) 111-136

\bibitem{}
C. Fefferman, C.R. Graham, \emph{Conformal invariants}, Asterisque (1985) 95-116

\bibitem{}
C. Fefferman, C.R. Graham, \emph{Q-curvature and Poincare metrics},  Math. Res. Lett. \textbf{9}(2002), 139-151

\bibitem{}
C. Fefferman, C.R. Graham, \emph{The ambient metric}, preprint, arXiv:0710.0919

\bibitem{}
C.R. Graham, \emph{Volume and area renormalizations for conformally compact Einstein metrics}, in: The Proceedings of the 19th Winter School
             ``Geometry and Physics'' 31-42
\bibitem{}
C.R. Graham, J.M. Lee, \emph{Einstein metrics with prescribed conformal infinity on the ball}, Adv. Math. \textbf{87}(1991), 186-225

\bibitem{}
C.R. Graham, E. Witten, \emph{Conformal anomaly of submanifold observables in AdS/CFT correspondence}, Nuclear Phys. B\textbf{546}(1999), 52-64

\bibitem{}
A. Gray, \emph{Tubes}, 2nd ed. Prog. Math. vol 221

\bibitem{}
M. Henningson, K. Skenderis, \emph{The holographic Weyl anomaly}, J. High Ehergy Phys. \textbf{7}(1998), Paper 23

\bibitem{}
M.L. Labbi, \emph{Double forms, curvature structures and the (p,q)-curvatures}, Trans. Amer. Math. Soc. \textbf{357} (2005), 3971-3992

\bibitem{}
M.L. Labbi, \emph{Variational properties of the Gauss-Bonnet curvatures}, Cal. Var. PDE \textbf{32} (2008), 175-189

\bibitem{}
 J.M. Lee, \emph{Fredholm operators and Einstein metrics on conformally compactmanifolds}, Mem. AMS 183 (864) (2006)

\bibitem{}
 D. Lovelock, \emph{The Einstein tensor and its generalizations}, J. Math. Phys.textbf{12} (1971), 498-501

\bibitem{}
S. Paycha, \emph{Anomalies and regularization techniques in mathematics and physics}, preprint

\bibitem{}
E. Witten, \emph{Anti de Sitter space and holography}, Adv. Theor. Math Phys. \textbf{2}(1998), 253-291


















\end{thebibliography}

\end{document}